\documentclass[
submission
 , nomarks
]{dmtcs-episciences}

\usepackage[utf8]{inputenc}
\usepackage{subfigure}
\usepackage{enumerate}

\usepackage[round]{natbib}
\usepackage{multirow} 
\newtheorem{remark}{Remark}
\newtheorem{theorem}{Theorem}
\newtheorem{lemma}{Lemma}

\newtheorem{corollary}{Corollary}

\newtheorem{pretheorema}{{\bf Theorem}}

\newtheorem{prelem}{{\bf Theorem}}

\newenvironment{lem}{\begin{prelem}{\hspace{-0.5
               mm}{\bf}}}{\end{prelem}}
\newtheorem{precor}{{\bf Corollary}}

\newtheorem{prelemlem}{{\bf Lemma}}

\newtheorem{preconstruction}{{\bf Construction}}

\newenvironment{construction}{\begin{preconstruction}{\hspace{-0.5
               mm}{\bf}}}{\end{preconstruction}}

\author{Hanieh Amjadi%\thanks{I am fully supported.}
  \and Nasrin Soltankhah\thanks{Corresponding author}}
\title[The 3-way flower intersection problem for Steiner triple systems]{The 3-way flower intersection problem for Steiner triple systems}
\affiliation{
  Faculty of Mathematical Sciences, Alzahra University, Tehran, Iran}
\received{2019-08-23}
\accepted{2020-01-09}
\begin{document}
\publicationdetails{22}{2020}{1}{5}{5706}
\maketitle
\begin{abstract}
The flower at a point $x$ in a Steiner triple system $(X,\mathcal{B})$ is the set of all triples containing $x$. Denote by $J^3_F(r)$ the set of all integers $k$ such that there exists a collection of three $STS(2r+1)$ mutually intersecting in the same set of $k+r$ triples, $r$ of them being the triples of a common flower. In this article we determine the set $J^3_F(r)$ for any positive integer $r\equiv 0,1$ (mod $3$) (only some cases are left undecided for $r =6,~7,~9,~24$), and establish that $J^3_F(r)=I^3_F (r)$ for $r\equiv 0,1$ (mod $3$) where $I^3_F(r)=\{0,1,\ldots ,\frac{2r(r-1)}{3}-8,\frac{2r(r-1)}{3}-6,\frac{2r(r-1)}{3}\}$.
\end{abstract}

\section{Introduction}
\label{sec:in}
A \emph{Steiner system} $S(2, k, v)$ is a pair $(X,\mathcal{B})$ where $X$ is a $v$-set and $\mathcal{B}$ is a family of $k$-subset of $X$ called blocks, such that each $2$-subset of $X$ is contained in exactly one block of $\mathcal{B}$.
If $k=3$, then the Steiner system $S(2,3,v)$ is called \emph{Steiner triple system} of order $v$ or briefly $STS(v)$. It is well known that an $STS(v)$ exists if and only if $v\equiv 1, 3$ (mod $6$) \cite{Hanani}.

A \emph{Kirkman triple system} of order $v$ (briefly $KTS(v)$) is a Steiner triple system of order $v$, $(X,\mathcal{B})$ together with a partition $\mathcal{R}$ of the set of triples $\mathcal{B}$ into subsets $\mathcal{R}_{1},\mathcal{R}_{2},\ldots ,\mathcal{R}_{n}$ called parallel classes such that each $\mathcal{R}_{i}~($for $i=1, 2, \dots , n)$ is a partition of $X$. It is well known that a $KTS(v)$ exists if and only if $v\equiv 3$ (mod $6$) \cite{Schoolgirl}.

It can be easily checked that the number of triples contained in an $STS(v)$ (or $KTS(v)$) is $t_{v}=v(v-1)/6$. For each non-negative integer $n$, let $S[n]$ denote the set of non-negative integers less than or equal to $n$, with the exception of $n-1$, $n-2$, $n-3$ and $n-5$. Let $I(v)=S[t_{v}]$.

Two Steiner (or Kirkman) triple systems $(X,\mathcal{B}_{1})$ and $(X,\mathcal{B}_{2})$ are said to \emph{intersect} in $k$ triples, provided $|\mathcal{B}_{1}\cap  \mathcal{B}_{2}|=k$. Denote by $J(v)$ (or $J_{R}(v)$) the set of all integer numbers $k$ such that there exists a pair of $STS(v)$ (or $KTS(v)$) intersecting in $k$ triples. 

\cite{2STS} have completely determined the set $J(v)$ and proved that $J(3)=1$, $J(7)=\{0, 1, 3, 7\}$, $J(9)=\{0, 1, 2, 3, 4, 6, 12\}$ and $J(v)=I(v)$ for $v\equiv 1, 3$ (mod $6$) and $v\geq 13$. Also \cite{2KTS-Shen}, \cite{2KTS} have determined the set $J_{R}(v)$.
 
For each non-negative integer $n$, let $S^3[n]$ denote the set of non-negative integers less than or equal to $n$, with the exception of $n-1$, $n-2$, $n-3$, $n-4$, $n-5$ and $n-7$. Let $I^{3}(v)=S^3[t_{v}]$.

\cite{3STS} generalized the intersection problem for $STS(v)$s and determined the set $J^{3}(v)$ of all integer numbers $k$ such that there exists a collection of three $STS(v)$s pairwise intersecting in the same set of $k$ blocks. The following  theorem contains their results.
\begin{lem}\emph{\cite{3STS}}{\label{3STS}}
$J^{3}(v)=I^{3}(v)$ for every $v\geq 19$; $J^{3}(7)=\{1,7\}$, $J^{3}(9)=\{0,1,3,4,12\}$, $J^{3}(13)=I^{3}(13)\setminus \{14,15,16,17,18,20\}$ and $J^{3}(15)=I^{3}(15)\setminus \{24,25,26,27\}$.
\end{lem}

\cite{3KTS} generalized the intersection problem for $KTS(v)$s and for sufficiently large $v$ determined the set $J^{3}_R(v)$ of all integer numbers $k$ such that there exists a collection of three $KTS(v)$s pairwise intersecting in the same set of $k$ blocks. Just two cases are left undecided. 

There are some other results on $\mu$-way intersection problem, for example, \cite{Rashidi-intersection} solved the $3$-way intersection problem of $S(2,4,v)$ designs also \cite{Milici86} found some results on the maximum number of $STS$s such that any two of them intersect in the same block-set and \cite{somi} found the maximum number of Steiner triple systems which intersect in some special blocks.

The \emph{flower} at a point $x$ in a Steiner triple system, is the set of all triples containing $x$. The \emph{flower intersection problem} for $STS$s is the determination for each $v=2r+1\equiv 1, 3$ (mod $6$) of the set $J_F(r)$ of all $k$ such that there exists a pair of Steiner triple systems on the same $v$-set having $k+r$ triples in common, $r$ of them being the triples of a common flower. Note that $r\equiv 0, 1$ (mod $3$); we call such a non-negative $r$ \emph{admissible}.

For each admissible $r$, let $I_F(r)=S[\frac{2r(r-1)}{3}]$. \cite{2FSTS} have completely determined the set $J_F(r)$. The following theorem contains their results.
\begin{lem}\emph{\cite{2FSTS}}{\label{2FSTS}}
For all admissible $r$, $J_F(r)=I_F(r)$, except $1,4\notin J_F(4)$.
\end{lem}
Also \cite{FS(24v)} solved the flower intersection problem of $S(2,4,v)$ designs; \cite{Milici90} solved the flower intersection problem for $S(3,4,v)$ designs for some values of $v$ and \cite{2FKTS} solved the flower intersection problem of $KTS(v)$s.

Let $I^3_F(r)=S^3[\frac{2r(r-1)}{3}]$ and $J^{3}_{F}(r)$ be the set of all integer numbers $k$ such that there exists a collection of three $STS(2r+1)$s pairwise intersecting in the same $k+r$ blocks, $r$ of them being the triples of a common flower. Here we determine $J^{3}_{F}(r)$ for all admissible values of $r$ except for some small values of $r$.

The necessary conditions are expressed in the following lemma and it is straightforward.
\begin{lemma}{\label{necessary conditions}}
For each admissible $r$, $J^{3}_{F}(r)\subseteq I^3_F(r)$, $J^{3}_{F}(r)\subseteq \{a-r |~ a\in J^3(2r+1),~a\geq r\}$ and $J^{3}_{F}(r)\subseteq J_F(r)$.
\end{lemma}
%%%%%%%%%%%%%%%%%%%%%%%%%%%%%%%%%%%%%%%%%%%%%%%%%%%%%%%%%%%%%%
%%%%%%%%%%%%%%%%%%%%%%%%%%%%%%%%%%%%%%%%%%%%%%%%%%%%%%%%%%%%%%
%%%%%%%%%%%%%%%%%%%%%%%%%%%%%%%%%%%%%%%%%%%%%%%%%%%%%%%%%%%%%%
\section{Recursive constructions}
In this section we give several recursive constructions for finding the $3$-way flower intersection numbers of Steiner triple systems. The concept of PBDs, GDDs and Latin squares play an important role in these constructions. Let us give their formal definition.

Let $v$ be a positive integer and $K$ be a set of positive integers. A \emph{pairwise balanced design} (briefly a PBD) $B(K, \lambda, v)$ is a pair $(X ,\mathcal{A})$ where $X$ is a $v$-set and $\mathcal{A}$ is a set of subsets (called \emph{blocks}) of $X$ such that $|B|\in K$ for each $B\in \mathcal{A}$ and each pair of distinct elements of $X$ is contained in exactly $\lambda$ blocks of $\mathcal{A}$.

Let $X$ be a finite set containing $v$ points, $\mathcal{G}$ a family of distinct subsets of $X$, called \emph{groups} which partition $X$, and $\mathcal{A}$ a collection of subsets of $X$, called \emph{blocks}. Let $K$ be a set of positive integers. A design $(X, \mathcal{G}, \mathcal{A})$ is called a \emph{group divisible design} (GDD) $K$-GDD if
\begin{enumerate}
\item $\{|B|~:~B\in \mathcal{A}\}\subseteq K$;
\item $|G\cap B|\leq 1$ for every $G\in \mathcal{G}$ and every $B\in \mathcal{A}$;
\item Every pair of points from distinct groups occurs in exactly a unique block of $\mathcal{A}$.
\end{enumerate}

If $\mathcal{G}$ contains $t_{1}$ groups of size $m_{1}$, $t_{2}$ groups of size $m_{2}$, \ldots, and $t_{s}$ groups of size $m_{s}$, we call $m^{t_{1}}_{1}m^{t_{2}}_{2}\ldots m^{t_{s}}_{s}$ the \emph{group type} (or type) of the GDD. A $K$-GDD with group type $m^{t_{1}}_{1}m^{t_{2}}_{2}\ldots m^{t_{s}}_{s}$ is actually a pairwise balanced design and it is denoted by $B(K\cup M, 1, v)$ (or $(v, K\cup M, 1)-$PBD), where $M=\{m_1,m_2,\ldots ,m_s\}$. We usually write $\{k\}$-GDD as $k$-GDD and $B(\{k\}, 1, v)$ as $B(k, 1, v)$ (or $(v, k, 1)$-BIBD).

A \emph{Latin square} of order $n$ is an $n\times n$ array $L = ({\ell_{ij}})$ on $n$ symbols in which every row and every column of $L$ contains no repeated symbols. Two Latin squares $L$ and $L'$ of the same order are \emph{orthogonal} if ${\ell}_{ab} = {\ell}_{cd}$ and ${\ell'}_{ab} = {\ell'}_{cd}$, implies $a = c$ and $b = d$.
A set of Latin squares $L_{1}, L_{2}, \ldots ,L_{m}$ are \emph{mutually orthogonal}, or a set of \emph{MOLS}, if for every $1 \leq i < j \leq m$, $L_{i}$ and $L_{j}$ are orthogonal.

We use the results of $3$-way intersection problem for Latin squares to determine $J^{3}_{F}(r)$. The 3-way intersection problem for Latin squares is the problem of determining, for all orders $n$, the set of integers $k$ for which there exist three Latin squares of order $n$ having precisely $k$ identical cells, with their remaining $n^2-k$ cells different in all three Latin squares.

Denote by $J'^3(n)$ the set of integers $k$ for which there exist three Latin squares of order $n$ which have precisely $k$ cells where all three squares have identical entries and $n^2-k$ cells where all three squares contain different entries. Let for $n\geq 4$, $I'^3(n)=[0,n^2-15]\cup \{n^2-12,n^2-9,n^2\}$ and $I'^3(3)=\{0,9\}$. \cite{3way-Latin} completely solve the $3$-way intersection problem. The following theorem contains their results. Let $[i,j]$ denote the set of integers $\{i,i + 1,i + 2,\ldots ,j-1,j\}$, where $i<j$.
\begin{lem}\emph{\cite{3way-Latin}}{\label{3way-Latin}}
$J'^3(3)=I'^3(3)=\{0,9\}$,
\\$J'^3(4)=I'^3(4)\setminus \{7\} =\{0,1,4,16\}$,
\\$J'^3(5)=I'^3(5)\setminus \{8,9,13,16\} =[0,7]\cup \{10,25\}$,
\\$J'^3(6)=I'^3(6)\setminus \{20,21,24\} =[0,19]\cup \{27,36\}$,
\\$J'^3(7)=I'^3(7)\setminus [31,34] =[0,30]\cup \{37,40,49\}$,
\\$J'^3(n)=I'^3(n)$ for $n\geq8$.
\end{lem}
Before expressing the next theorem we need to introduce some notations. For sets of integers $X$ and $Y$, let $X + Y =\{x + y|~x\in X,~y\in Y\}$. Also for any integer $n$, let $nX =\{nx|~x\in X\}$.

Let $J''^3(n)$ denotes the set of integers $k$ for which there exist three Latin squares of order $n$ with the same constant secondary diagonal which have precisely $k=k'+n$ cells ($n$ of these are cells of secondary diagonal) where all three squares have identical entries and $n^2-n-k'$ cells where all three squares contain different entries. The following theorem is similar to Lemma $2.2$ of \cite{3way-Latin}.
\begin{theorem}{\label{Latin J''^3(2n+1)}}
$J'^3(n)+(n+1)\{[0,n-3]\cup \{n\}\}+(2n+1)\{[1,n-2]\cup \{n+1\}\}\subseteq J''^3(2n+1).$
\end{theorem}
\begin{proof}
The proof is exactly the same as described in the Lemma $2.2$ of \cite{3way-Latin}, just it is enough to use permutations which do not change the secondary diagonal of the Latin square.
\end{proof}

We use four constructions of $STS$s from smaller $STS$s which are adapted from constructions used in \cite{Packing-Billington}, \cite{2FSTS} and \cite{2STS}. 
\begin{construction}{\label{STS(3v)-from-STS(v)}}%\cite{Packing-Billington}
\emph{($STS(3v)$ from $STS(v)$)}. We define Steiner triple system $S$ on the set $X=\{1,2, \ldots , v\}\times \{1,2,3\}$. Let $L=(\ell_{ij})$ be a Latin square on symbols $V=\{1,2, \ldots , v\}$. Let $V_i=V\times \{i\}$ for $i=1,2,3$ and $(V_1,\mathcal{B}_1)$, $(V_2,\mathcal{B}_2)$ and $(V_3,\mathcal{B}_3)$ be three Steiner triple systems $S_1$, $S_2$ and $S_3$ respectively. It is easy to see that $(X,\mathcal{B}_1\cup \mathcal{B}_2\cup \mathcal{B}_3\cup \mathcal{B})$ is an $STS(3v)$, where $\mathcal{B}=\{\{(i,1),(j,2),(s,3)\}|~\ell_{ij}=s,~1\leq i,j,s\leq v\}$.
\end{construction}

\begin{construction}{\label{STS(3v+1)-from-STS(v+1)}}%\cite{2FSTS}
\emph{($STS(3v+1)$ from $STS(v+1)$)}. We define Steiner triple system $S$ on the set $X=(\{1,2, \ldots , v\}\times \{1,2,3\})\cup \{\infty\}$. Let $L$ be a Latin square on symbols $V=\{1,2, \ldots , v\}$. Let $V_i=V\times \{i\}$ for $i=1,2,3$ and $(V_1\cup \{\infty\},\mathcal{B}_1)$, $(V_2\cup \{\infty\},\mathcal{B}_2)$ and $(V_3\cup \{\infty\},\mathcal{B}_3)$ be three $STS(v+1)$ $S_1$, $S_2$ and $S_3$ respectively. It is easy to see that $(X,\mathcal{B}_1\cup \mathcal{B}_2\cup \mathcal{B}_3\cup \mathcal{B})$ is an $STS(3v+1)$, where $\mathcal{B}=\{\{(i,1),(j,2),(s,3)\}|~\ell_{ij}=s,~1\leq i,j,s\leq v\}$.
\end{construction}

\begin{construction}{\label{STS(3v+3)-from-STS(v+3)}}
\emph{($STS(3v+3)$ from $STS(v+3)$)}. We define Steiner triple system $S$ on the set $X=(\{1,2, \ldots , v\}\times \{1,2,3\})\cup \{\infty_1,\infty_2,\infty_3\}$. Let $L=(\ell_{ij})$ be a Latin square on symbols $V=\{1,2, \ldots , v\}$. Let $V_i=V\times \{i\}$ for $i=1,2,3$ and $(V_1\cup \{\infty_1,\infty_2,\infty_3\},\mathcal{B}_1)$, $(V_2\cup \{\infty_1,\infty_2,\infty_3\},\mathcal{B}_2)$ and $(V_3\cup \{\infty_1,\infty_2,\infty_3\},\mathcal{B}_3)$ be three $STS(v+3)$ $S_1$, $S_2$ and $S_3$ respectively, which each of them contains triple $\{\infty_1,\infty_2,\infty_3\}$. It is easy to see that $(X,\mathcal{B}_1\cup \mathcal{B}_2\cup \mathcal{B}_3\cup \mathcal{B})$ is an $STS(3v+3)$, where $\mathcal{B}=\{\{(i,1),(j,2),(s,3)\}|~\ell_{ij}=s,~1\leq i,j,s\leq v\}$.
\end{construction}

Recall that a $1$-factor of a graph $G$ is a spanning subgraph of G that is regular of degree $1$. A $1$-factorization of a graph $G$ is a set $\mathcal{F}=\{F_1,\ldots ,F_k\}$ of edge-disjoint $1$-factors of $G$ whose edge-sets partition the edge-set of $G$. \cite{Handbook}

\begin{construction}{\label{STS(2v+1)-from-STS(v)}}
\emph{($STS(2v+1)$ from $STS(v)$)}. Let $X=\{a_1,a_2,\ldots ,a_v\}$ and $(X,\mathcal{B})$ be an $STS(v)$. Put $v + 1 = 2n$ and let $\mathcal{F}=\{F_i|~i=1, 2, . . . , 2n-1\}$ be a $1$-factorization of $K_{2n}$ with the vertex set $V(K_{2n}) = Y$ where $X\cap Y = \emptyset$, ($K_{2n}$ is complete graph with $2n$ vertices). Put $X^* = X\cup Y$ and $\mathcal{B}^* = \mathcal{B}\cup \mathcal{C}$ where $\mathcal{C}= \{\{a_i,x,y\}|~[x, y]\in F_i,~i = 1, 2, . . . , 2n-1\}$. It is easy to see that $(X^*,\mathcal{B}^*)$ is an $STS(2v + 1)$.
\end{construction}
\begin{theorem}{\label{STS(3v) Construction}}
For all admissible $r$, if there exist three Latin squares on same $2r+1$ symbols with the same constant secondary diagonal agreeing pairwise on exactly $b+2r+1$ cells ($2r+1$ of these are cells of the secondary diagonal), and if $a_i\in J^3(2r+1)$ for $i=1,2$ and $a\in J^3_F(r)$, then $a_1+a_2+a+b\in J^3_F(3r+1)$.
\end{theorem}
\begin{proof}
We use Construction {\ref{STS(3v)-from-STS(v)}} to construct a collection of three $STS(6r+3)$s with flower intersection number $a_1+a_2+a+b$. For this purpose, start with three Latin squares $L'$, $L''$ and $L'''$ of order $2r+1$, with the same constant secondary diagonal of $1$'s and agreeing pairwise on exactly $b$ cells off the secondary diagonal.  For $i = 1,2$, let $S'_i$, $S''_i$ and $S'''_i$ be a collection of three systems on $\{1,2, \ldots , 2r+1\}\times \{i\}$ with exactly $a_i$ triples in common. Let $S'_3$, $S''_3$ and $S'''_3$ be a collection of three systems on $\{1,2, \ldots , 2r+1\}\times \{3\}$ with the same $r$ triples containing $(1,3)$ (a flower at element $(1,3)$), and $a$ further triples in common. Then by using Construction {\ref{STS(3v)-from-STS(v)}} for $v=2r+1$, we construct a collection of three Steiner triple systems $S'$, $S''$ and $S'''$ of order $6r+3$ with $a_1+a_2+a+b+(3r+1)$ triples in common, $3r+1$ of them are flower at element $(1,3)$.  
\end{proof}
\begin{remark}{\label{Remark-STS(3v) Construction}}
Theorem \ref{STS(3v) Construction} is true also when Latin squares have same row or same column instead of same constant secondary diagonal. the proof is similar to the proof of Theorem \ref{STS(3v) Construction}.
\end{remark}
\begin{theorem}{\label{STS(3v+1) Construction}}
For all admissible $r$, if $b\in J'^3(2r)$ and $a_i\in J^3_F(r)$ for $i=1,2,3$, then $a_1+a_2+a_3+b\in J^3_F(3r)$.
\end{theorem}
\begin{proof}
We use Construction {\ref{STS(3v+1)-from-STS(v+1)}} to construct a collection of three $STS(6r+1)$s with flower intersection number $a_1+a_2+a_3+b$. For this purpose, start with three Latin squares $L'$, $L''$ and $L'''$ of order $2r$ with $3$-way intersection number $b$. For $i = 1,2,3$, let $S'_i$, $S''_i$ and $S'''_i$ be a collection of three systems on $(\{1,2, \ldots , 2r\}\times \{i\})\cup \{\infty\}$ with the same $r$ triples containing $\infty$ (a flower at $\infty$), and $a_i$ further triples in common. Then by using Construction {\ref{STS(3v+1)-from-STS(v+1)}} for $v=2r$, we construct a collection of three Steiner triple systems $S'$, $S''$ and $S'''$ of order $6r+1$ with $a_1+a_2+a_3+b+3r$ triples in common, $3r$ of them are flower at $\infty$.  
\end{proof}
\begin{theorem}{\label{STS(3v+3) Construction}}
For $r\equiv 0, 2$ (mod $3$), if $b\in J'^3(2r)$ and $a_i\in J^3_F(r+1)$ for $i=1,2,3$, then $a_1+a_2+a_3+b\in J^3_F(3r+1)$.
\end{theorem}
\begin{proof}
We use Construction {\ref{STS(3v+3)-from-STS(v+3)}} to construct a collection of three $STS(6r+3)$s with flower intersection number $a_1+a_2+a_3+b$. For this purpose, start with three Latin squares $L'$, $L''$ and $L'''$ of order $2r$ with $3$-way intersection number $b$. For $i = 1,2,3$, let $S'_i$, $S''_i$ and $S'''_i$ be a collection of three systems on $(\{1,2, \ldots , 2r\}\times \{i\})\cup \{\infty_1,\infty_2,\infty_3\}$ which contain triple $\{\infty_1,\infty_2,\infty_3\}$ and with the same $r+1$ triples containing $\infty_1$ (a flower at $\infty_1$), and $a_i$ further triples in common. Then by using Construction {\ref{STS(3v+3)-from-STS(v+3)}} for $v=2r$, we construct a collection of three Steiner triple systems $S'$, $S''$ and $S'''$ of order $6r+3$ with $a_1+a_2+a_3+b+(3r+1)$ triples in common, $3r+1$ of them are flower at $\infty_1$.  
\end{proof}
\begin{theorem}{\label{STS(2v+1) Construction}}
For all admissible $r$, if $k\in J^3_F(r)$ then $k+(s-1)(r+1)\in J^3_F(2r+1)$ for every $s=1,2,\ldots ,2r-2,2r+1$.       
\end{theorem}
\begin{proof}             
Let $v=2r+1$ and $X$, $Y$, $X^*$ and $\mathcal{F}$ be as in Construction \ref{STS(2v+1)-from-STS(v)}. Let $(X,\mathcal{B}_i)$ for $i=1,2,3$ be a collection of three $STS(v)$s, with $3$-way intersection number $k+r$ which $r$ of them being the triples of a common flower at point $a_1$. For $s=1,2,\ldots ,v-3,v$, let $\alpha_1$ and $\alpha_2$ be two permutations of $X$ fixing exactly $s$ elements $\{a_1,a_2,\ldots ,a_s\}$ and for $s+1\leq i\leq v$, $\alpha_1(a_i)\neq \alpha_2(a_i)$. Now let $\mathcal{C}$ be as in Construction \ref{STS(2v+1)-from-STS(v)}, i.e.

$\mathcal{C}= \{\{a_i,x,y\}|~[x, y]\in F_i,~i = 1, 2, . . . , v\}$, and put

$\alpha_1\mathcal{(C)}= \{\{\alpha_1(a_i),x,y\}|~[x, y]\in F_i,~i = 1, 2, . . . , v\}$, and

$\alpha_2(\mathcal{C})= \{\{\alpha_2(a_i),x,y\}|~[x, y]\in F_i,~i = 1, 2, . . . , v\}$.

Each $1$-factor of $\mathcal{F}$ contains $(v+1)/2=r+1$ edges, so $\mathcal{C}$, $\alpha_1\mathcal{(C)}$ and $\alpha_2\mathcal{(C)}$ pairwise have exactly $(s-1)(r+1)+(r+1)$ triples in common, which $(r+1)$ of them being the triples which contain $a_1$. So three $STS(2v+1)$s, $(X^*,\mathcal{B}_1\cup \mathcal{C})$, $(X^*,\mathcal{B}_2\cup \alpha_1(\mathcal{C}))$ and $(X^*,\mathcal{B}_3\cup \alpha_2(\mathcal{C}))$ have $3$-way intersection number $k+(s-1)(r+1)+2r+1$ which $2r+1$ of them being the triples of a common flower at point $a_1$. So $k+(s-1)(r+1)\in J^3_F(2r+1)$.               
\end{proof}
\begin{lem}\emph{\cite{Handbook}(Section III-3-4),\cite{4MOLS14}}{\label{4Mols}}
For $n\geq 5$ and $n\neq 6$, there exist four mutually orthogonal Latin squares of order $n$ except possibly for $n\in \{10,18,22\}$.
\end{lem}
\begin{theorem}{\label{6t-PBD}}
Let $t\geq 5$ and $t\notin \{6,10,18,22\}$, there exists a pairwise balanced design (PBD) of order $6t$, with six blocks of size $t$ and $t^2$ blocks of size $6$. (or a $B(\{6,t\},1,6t)$).
\end{theorem}
\begin{proof}
Let $V=\{1,2,\ldots ,t\}$, we define a pairwise balanced design $B(\{6,t\},1,6t)$ on set $V\times \{1,2,\ldots ,6\}$. By Theorem \ref{4Mols} there exist four mutually orthogonal Latin squares of order $t$, where $t\geq 5$ and $t\notin \{6,10,18,22\}$. Let $L_1$, $L_2$, $L_3$ and $L_4$ be four mutually orthogonal Latin squares of order $t$, where $L_n=(\ell^{\{n\}}_{ij})$, $i\in V\times \{5\}, j\in V\times \{6\}$ and $\ell^{\{n\}}_{ij}\in V\times \{n\}$ for $n\in \{1,2,3,4\}$. It is enough to consider six blocks $\{V\times \{m\}\}$ for $m\in \{1,2,\ldots ,6\}$ and $t^2$ blocks $\{\{a_1,a_2,a_3,a_4,i,j\}|~a_n\in V\times \{n\},~1\leq n\leq 4,~i\in V\times \{5\}, j\in V\times \{6\}\}$, where $\ell^{\{n\}}_{ij}=a_n$ for $1\leq n\leq 4$.
\end{proof}

The following theorem is similar to Theorem $4.1.3$ of \cite{Lindner-Rodger}.
\begin{theorem}{\label{STS-from-PBD}}
If there exists a $B(\{k_1,k_2,\ldots ,k_x\},1,r)$ for $r\equiv 0,1$ (mod $3$), and if there exists an $STS(2k_i+1)$ for $1\leq i\leq x $, then there exists an $STS(2r+1)$.
\end{theorem}
\begin{proof}
Let $(X,\mathcal{B})$ be a $B(\{k_1,k_2,\ldots ,k_x\},1,r)$ with $X=\{1,2,\ldots ,r\}$. Define an $STS(2r+1)$ $(Y,\mathcal{A})$ with $Y=\{\infty\}\cup (\{1,2,\ldots ,r\}\times \{1,2\})$ as follows.
\begin{itemize}
\item[(1)] For $1\leq i\leq r$, $\{\infty ,(i,1),(i,2)\}\in \mathcal{A}$, 
\item[(2)] For each block $B\in \mathcal{B}$, let $(Y(B),\mathcal{A}(B))$ be an $STS(2|B|+1)$, where $Y(B)=\{\infty\}\cup (B\times \{1,2\})$ and where the symbols have been named so that $\{\infty , (i,1),(i,2)\}\in \mathcal{A}(B)$ for all $i\in B$, and let $\mathcal{A}(B)\setminus \{\{\infty , (i,1),(i,2)\}|~i\in B\}\subseteq \mathcal{A}$.
\end{itemize}
It is easy to see that $(Y,\mathcal{A})$ is an $STS(2r+1)$.
\end{proof}
\begin{theorem}{\label{PBD Construction}}
Let $k\equiv 0,1$ (mod $3$) for each $k\in K$ and $r\equiv 0,1$ (mod $3$). If there exists a $B(K, 1, r)$ $(X, \mathcal{B})$ such that $k_{B}\in J^{3}_{F}(|B|)$ for each $B\in \mathcal{B}$, then $\sum_{B\in \mathcal{B}}k_{B}\in J^{3}_{F}(r)$.
\end{theorem}
\begin{proof}
Since $r\equiv 0,1$ (mod $3$) and for each $k\in K$, $k\equiv 0,1$ (mod $3$), we can use the construction which has been explained in Theorem {\ref{STS-from-PBD}} to form a collection of three $STS(2r + 1)$s $(Y, \mathcal{A}_1)$, $(Y,\mathcal{A}_2)$ and $(Y, \mathcal{A}_3)$ as follows. For each $B\in \mathcal{B}$, since $k_B\in J^{3}_{F}(|B|)$, then we may form a collection of three $STS(2|B|+1)$s $(Y(B), \mathcal{A}_1(B))$, $(Y(B), \mathcal{A}_2(B))$ and $(Y(B), \mathcal{A}_3(B))$ such that their $3$-way flower intersection number is $k_B$ and with flower at $\infty$. For $i =1, 2, 3$, we define
\\$\mathcal{A}_i =\{\{\infty , x_1 ,x_2\}|~x\in B\}\cup \{\cup_{B\in \mathcal{B}}\{\mathcal{A}_i(B)\setminus \{\{\infty , x_1 , x_2\}|~x\in B\}\}$.

So $(Y,\mathcal{A}_1)$, $(Y, \mathcal{A}_2)$ and $(Y,\mathcal{A}_3)$ are three $STS(2r + 1)$s with $3$-way flower intersection number $\sum_{B\in \mathcal{B}}k_{B}$ and the flower is on $\infty$. So $\sum_{B\in \mathcal{B}}k_B\in J^{3}_{F}(r)$.
\end{proof}

In the following, some auxiliary lemmas are expressed.
\begin{lemma}{\label{0-in-J^3_F}}
For all admissible $r\geq 4$, $0\in J^3_F(r)$.
\end{lemma}
\begin{proof}
Let $(X,\mathcal{B})$ be an $STS(2r+1)$ and write $\mathcal{B}=\mathcal{F}_x\cup \mathcal{C}$, where $\mathcal{F}_x$ is the flower at the point $x$. Let $Y=X-\{x\}$ and $\mathcal{G}=\{\{a,b\}:~\{a,b,x\}\in \mathcal{F}_x\}$. Then $(Y,\mathcal{G},\mathcal{C})$ is a $3$-GDD of type $2^r$. Since the maximum number of disjoint $3$-GDDs of type $2^r$ is $2(r-2)$ and there exists a large set of $3$-GDD of type $2^r$ for $r\equiv 0, 1$ (mod $3$) \cite{LargeSet-GDD}, the result follows.
\end{proof}
\begin{lemma}{\label{2r(r-1)/3-in-J^3_F}}
For all admissible $r$, $\frac{2r(r-1)}{3}\in J^3_F(r)$.
\end{lemma}
\begin{proof}
It is enough to consider the same $STS(2r+1)$ three times.
\end{proof}

The following lemma is similar to Lemma $1$ of \cite{2FSTS}.
\begin{lemma}{\label{lem-k>(2r(r-3)/3)}}
Let $v=2r+1\equiv 1, 3$ (mod $6$), for $r\geq6$ and $k\geq\frac{2r(r-3)}{3}$, if $k+r\in J^3(v)$, then $k\in J^3_F(r)$. 
\end{lemma}
\begin{proof}
It is enough to assume a point $x$ is contained in a triple of the first system that is not a triple of the other systems, then obviously it must be contained in at least two such triples. A simple calculation now shows that there must be at least one point $x$ for which the triples through $x$ are the same in all systems.
\end{proof}
\begin{corollary}{\label{cor-k>(2r(r-3)/3)}}
For all admissible $r\geq 9$, let $k\in I^3_F(r)$ with $k\geq\frac{2r(r-3)}{3}$, then $k\in J^3_F(r)$.
\end{corollary}
\begin{proof}
From \cite{3STS}, $J^3(v)=I^3(v)$ for every $v\geq 19$. Lemma \ref{lem-k>(2r(r-3)/3)} completes the proof.
\end{proof}
%%%%%%%%%%%%%%%%%%%%%%%%%%%%%%%%%%%%%%%%%%%%%%%%%%%%%%%%%%%%
%%%%%%%%%%%%%%%%%%%%%%%%%%%%%%%%%%%%%%%%%%%%%%%%%%%%%%%%%%%%
%%%%%%%%%%%%%%%%%%%%%%%%%%%%%%%%%%%%%%%%%%%%%%%%%%%%%%%%%%%%
%%%%%%%%%%%%%%%%%%%%%%%%%%%%%%%%%%%%%%%%%%%%%%%%%%%%%%%%%%%%
\section{Small Cases}
In this section we discuss some small admissible values of $r$, needed for general constructions.

For a $v$-set $X$, let $(X,\mathcal{A})$, $(X,\mathcal{B})$ and $(X,\mathcal{C})$ be three Steiner triple systems of order $v$. For convenience, we introduce a notation in this section. $|\mathcal{A}\cap \mathcal{B}\cap \mathcal{C}|_F=k$, means that three $STS(v)$s, $(X,\mathcal{A})$, $(X,\mathcal{B})$ and $(X,\mathcal{C})$, have $3$-way flower intersection number $k$.

The following theorem is obvious.
\begin{theorem}{\label{J^3_F(1)}}
$J^3_F(1)=\{0\}$.
\end{theorem}
\begin{theorem}{\label{J^3_F(3)}}
$J^3_F(3)=\{4\}$.
\end{theorem}
\begin{proof}
By Theorem \ref{3STS}, $J^3(7)=\{1,7\}$, so the only $3$-way flower intersection number of $STS(7)$ can be $4$. Using Lemma {\ref{2r(r-1)/3-in-J^3_F}} completes the proof.
\end{proof}
\begin{theorem}{\label{J^3_F(4)}}
$J^3_F(4)=\{0,8\}$.
\end{theorem}
\begin{proof}
By Theorem \ref{3STS}, $J^3(9)=\{0,1,3,4,12\}$, so the only $3$-way flower intersection numbers of $STS(9)$ can be $\{0,8\}$. Lemmas {\ref{0-in-J^3_F}} and {\ref{2r(r-1)/3-in-J^3_F}} completes the proof. 
\end{proof}
\begin{theorem}{\label{J^3_F(6)}}
$[0,5]\cup \{7,20\}\subseteq J^3_F(6)\subseteq [0,7]\cup \{20\}$.
\end{theorem}
\begin{proof}
By Theorem \ref{3STS}, $J^{3}(13)=I^{3}(13)\setminus \{14,15,16,17,18,20\}$, so $J^3_F(6)\subseteq [0,7]\cup \{20\}$. By Lemmas \ref{0-in-J^3_F} and \ref{2r(r-1)/3-in-J^3_F}, $\{0,20\}\subseteq J^3_F(6)$. In \cite{3STS}, the $STS(13)$s which are given in Tables $1-7,~1-8$ and $1-9$ have $3$-way flower intersection numbers $2,~3$ and $4$ respectively with flower at $1$. Also the $STS(13)$s which are given in Table $1-11$  have $3$-way flower intersection number $7$ with flower at $13$. Let  $X=\{1,2,\ldots, 13\} $  and  $(X, \mathcal{A})$ and $(X, \mathcal{B})$ be the following $STS(13)$s.
\begin{center}
\begin{small}
$
\begin{array}{c c c c c c c c c c }
                       & ~1~~~2~11    & ~1~~~3~13 & ~1~~~4~10    & ~1~~~5~~~7 & ~1~~~6~12 & ~1~~~8~~~9 & ~2~~~3~~~6 & ~2~~~4~~~9 & ~2~~~5~12 \\
\mathcal{A}      & ~2~~7~13      & ~2~~~8~10  & ~3~~~4~~~7 & ~3~~~5~10    & ~3~~~8~12 & ~3~~~9~11   & ~4~~~5~~~8 & ~4~~~6~11       & ~4~12~13    \\
                       & ~5~~~6~~~9 & ~5~11~13    &   ~6~~~7~10  & ~6~~~8~13     & ~7~~~8~11 & ~7~~~9~12   & ~9~10~13   & ~10~11~12         &                      \\   
\end{array}
$
\end{small}
\end{center}
\begin{center}
\begin{small}
$
\begin{array}{c c c c c c c c c c c}
                        &~1~~~2~11     & ~1~~~3~13 & ~1~~~4~10    &  ~1~~~5~~~7 & ~1~~~6~12  & ~1~~~8~~~9 & ~2~~~3~~~6  &  ~2~~~4~13   &  ~2~~~5~12 \\
\mathcal{B}       & ~2~~~7~~~9 & ~2~~~8~10 & ~3~~~4~~~7 & ~3~~~5~10     & ~3~~~8~12  & ~3~~~9~11    & ~4~~~5 ~~~8 & ~4~~~6~11    & ~4~~~9~12 \\
                        &~5~~~6~~~9  & ~5~11~13    & ~6~~~7~10    & ~6~~~8~13     & ~7~~~8~11  & ~7~12~13       & ~9~10~13        & ~10~11~12     &                      \\
\end{array}
$
\end{small}
\end{center}
Consider the following permutations on $X$:

$
\begin{array}{l l}
\pi_1=(3~~4~~7)(13~~10~~5), & \pi_2=(13~~8)(12~~4)(2~~11)(3~~9)(6~~10), \\
\pi'_1=(7~~4~~3)(5~~10~~13), & \pi'_2=(12~~11~~9)(2~~8~~6)(4~~10)(5~~7).\\
\end{array}
$

It is checked by computer programming that $|\mathcal{A}\cap \pi_1(\mathcal{A})\cap \pi'_1(\mathcal{A})|_F=1$ and $|\mathcal{B}\cap \pi_2(\mathcal{A})\cap \pi'_2(\mathcal{A})|_F=5$.
\end{proof}

The following construction is similar to construction used in Lemma $4$ of \cite{2FSTS}.
\\Let us call the following Latin square, $L$, of order $8$ on symbols $0$, $1$, \ldots , $7$, \emph{special} of order $8$, where $A$ is a Latin square of order $4$ on symbols $4$, $5$, $6$ and $7$.
\begin{small}
\begin{center}$L=$
\begin{tabular}{|c c | c c | c c | c c|}
\hline
$0$         & $1$         & $2$         & $3$          & \multicolumn{4}{c|}{\multirow{4}{*}{$A$}}\\
$1$         & $0$         & $3$         & $2$          & \multicolumn{4}{c|}{}                              \\ \cline{1-4}
$2$         & $3$         & $0$         & $1$          & \multicolumn{4}{c|}{}                              \\
$3$         & $2$         & $1$         & $0$          & \multicolumn{4}{c|}{}                              \\ \hline
\multicolumn{4}{|c|}{\multirow{4}{*}{$A^T$}} & $0$         & $1$         & $2$         & $3$    \\
\multicolumn{4}{|c|}{}                                   & $1$         & $0$         & $3$         & $2$    \\ \cline{5-8}
\multicolumn{4}{|c|}{}                                   & $2$         & $3$         & $0$         & $1$    \\
\multicolumn{4}{|c|}{}                                   & $3$         & $2$         & $1$         & $0$    \\ \hline
\end{tabular}
\end{center}
\end{small}
Denote by $K$ the set of integers $k$ for which there exists a collection of three special Latin squares of order $8$ which pairwise agree in exactly $k$ of the $24$ cells above the $2\times 2$ diagonal blocks.
\begin{lemma}{\label{Latin-8-Intersection}}
$\{8,9,12\}\subseteq K$.
\end{lemma}
\begin{proof}
Let $k\in K$, write $k=a+b$ where $a\in J'^3(4)$ and $b=8$ (the number of cells which contain elements $2$ and $3$). By Theorem \ref{3way-Latin}, $\{0,1,4\}\subseteq J'^3(4)$, so $\{8,9,12\}\subseteq K$. 
\end{proof}
\begin{lemma}{\label{STS(15)12,13,16}}
$\{12,13,16\}\subseteq J^3_F(7)$.
\end{lemma}
\begin{proof}
Let $K$ be the set of integers $k$ for which there exists a collection of three special Latin squares of order $8$, $L^{(n)}_{ij}$ for $n=1,2,3$, which pairwise agree in exactly $k$ of the $24$ cells above the $2\times 2$ diagonal blocks. We construct a collection of three $STS(15)$s with $3$-way flower intersection number $h$ for $h\in \{12,13,16\}$. Write $h=\ell+k$ where $\ell\in J^3_F(3)$ and $k\in K$. Let $X_1=\{\infty_i:~1\leq i\leq 7\}$, $X_2=\{1,2,\ldots ,8\}$ and $(X_1,\mathcal{B}_n)$ for $n=1,2,3$ be three $STS(7)$s with $3$-way intersection number $\ell +3$ where three of these will constitute the flower at $\infty_1$. It is not hard to check that $(X_1\cup X_2,\mathcal{B}_n\cup \mathcal{C}_n)$ for $n=1,2,3$ are three $STS(15)$s, where $\mathcal{C}_n=\{\{\infty_s,i,j\}|~\ell^{(n)}_{ij}=s,~1\leq i < j\leq 8,~1\leq s\leq 7\}$ which have $3$-way intersection number $\ell +k+7$, seven of these will constitute the flower at $\infty_1$. So $h=\ell +k\in J^3_F(7)$. By Theorem \ref{J^3_F(3)}, $\ell =4$ and by Lemma \ref{Latin-8-Intersection}, $\{8,9,12\}\subseteq K$. So $\{12,13,16\}\subseteq J^3_F(7)$.
\end{proof}
\begin{theorem}{\label{J^3_F(7)}}
$[0,8]\cup [10,13]\cup \{16,22,28\}\subseteq J^3_F(7)\subseteq [0,16]\cup \{22,28\}$.
\end{theorem}
\begin{proof}
By Theorem \ref{3STS}, $J^{3}(15)=I^{3}(15)\setminus \{24,25,26,27\}$, so $J^3_F(7)\subseteq [0,16]\cup \{22,28\}$. By Lemmas \ref{0-in-J^3_F}, \ref{2r(r-1)/3-in-J^3_F}, \ref{lem-k>(2r(r-3)/3)} and \ref{STS(15)12,13,16}, $\{0,12,13,16,22,28\}\subseteq J^3_F(7)$. In Lemma $2$ of \cite{3KTS}, the three $KTS(15)$s which have intersection numbers $15$ and $17$ actually have $3$-way flower intersection numbers $8$ and $10$ respectively with a common flower at element $1$. Also in Lemma $3-4$ of \cite{3STS}, the three $STS(15)$s which have intersection numbers $12$ and $18$ actually have $3$-way flower intersection numbers $5$ and $11$ respectively with a common flower at element $4$ (Tables $2-1$, $2-4$ and $2-5$). Now let $X=\{1,2,\ldots, 15\} $  and  $(X, \mathcal{A})$, $(X, \mathcal{B})$, $(X, \mathcal{C})$, $(X, \mathcal{D})$ and $(X, \mathcal{E})$ be the following $STS(15)$s.
\begin{small}
\begin{center}
$
\begin{array}{c c c c c c c c}
& ~1~~~2~~~3  & ~1~~~4~~~5 &  ~1~~~6~~~7 &  ~1~~~8~~~9 &  ~1~10~11      &  ~1~12~13      &  ~1~14~15\\
& ~4~10~14        & ~2~~~9~11    & ~2~~~8~10     & ~2~13~15        & ~2~12~14       & ~2~~~4~~~6 & ~2~~~5~~~ 7\\
\mathcal{A}
& ~5 ~~~8 ~13   & ~3~~~8~12    & ~3~~~9~14     & ~3~~~4~~~7  & ~3~~~5~~~6 & ~3~11~15       & ~3~10~13\\
& ~6~~~9~15     & ~6~13~14       & ~4~11~13        & ~5~11~14        & ~4~~~8~15    & ~5~~~9~10    & ~4~~~9~12\\
& ~7~11~12        & ~7~10~15       & ~5~12~15        & ~6~10~12        & ~7~~~9~13    & ~7~~~8~14    & ~6~~~8~11\\
\end{array}
$

$
\begin{array}{c c c c c c c c}
& ~1~~~2~~~3  & ~1~~~4~~~5 &  ~1~~~6~~~7 &  ~1~~~8~~~9 &  ~1~10~11      &  ~1~12~13      &  ~1~14~15\\
& ~4~10~15        & ~2~~~9~11    & ~2~~~8~10     & ~2~13~15        & ~2~12~14       & ~2~~~4~~~6 & ~2~~~5~~~ 7\\
\mathcal{B}
& ~5 ~~~9 ~13   & ~3~~~8~11    & ~3~~~9~10     & ~3~~~4~~~7  & ~3~~~5~~~6 & ~3~13~14       & ~3~12~15\\
& ~6~~~9~15     & ~6~~~8~13    & ~4~11~13        & ~5~11~14        & ~4~~~8~12    & ~5~10~12    & ~4~~~9~14\\
& ~7~11~15        & ~7~10~13       & ~5~~~8~15        & ~6~10~14        & ~7~~~9~12    & ~7~~~8~14    & ~6~11~12\\
\end{array}
$

$
\begin{array}{c c c c c c c c}
& ~1~~~2~~~3  & ~1~~~4~~~5 &  ~1~~~6~~~7 &  ~1~~~8~~~9 &  ~1~10~11      &  ~1~12~13      &  ~1~14~15\\
& ~4~~~8~15     & ~2~~~9~11    & ~2~~~8~10     & ~2~13~15        & ~2~12~14       & ~2~~~4~~~6 & ~2~~~5~~~ 7\\
\mathcal{C}
& ~5 ~~~9 ~13   & ~3~~~8~12    & ~3~~~9~14     & ~3~~~4~~~7  & ~3~~~5~~~6 & ~3~11~13       & ~3~10~15\\
& ~6~~~9~12     & ~6~~~8~14    & ~4~13~14        & ~5~10~14        & ~4~11~12    & ~5~12~15    & ~4~~~9~10\\
& ~7~11~14        & ~7~10~12       & ~5~~~8~11        & ~6~10~13        & ~7~~~9~15    & ~7~~~8~13    & ~6~11~15\\
\end{array}
$

$
\begin{array}{c c c c c c c c}
& ~1~~~2~~~3  & ~1~~~4~~~5 &  ~1~~~6~~~7 &  ~1~~~8~~~9 &  ~1~10~11      &  ~1~12~13      &  ~1~14~15\\
& ~4~~~8~13     & ~2~~~9~11    & ~2~~~8~10     & ~2~13~15        & ~2~12~14       & ~2~~~4~~~6 & ~2~~~5~~~ 7\\
\mathcal{D}
& ~5 ~~~9 ~14   & ~3~~~8~11    & ~3~~~9~12     & ~3~~~4~~~7  & ~3~~~5~~~6 & ~3~13~14       & ~3~10~15\\
& ~6~~~9~15     & ~6~~~8~14    & ~4~11~14        & ~5~11~13        & ~4~12~15    & ~5~10~12    & ~4~~~9~10\\
& ~7~11~15        & ~7~10~14       & ~5~~~8~15        & ~6~10~13        & ~7~~~9~13    & ~7~~~8~12    & ~6~11~12\\
\end{array}
$

$
\begin{array}{c c c c c c c c}
& ~1~~~2~~~3  & ~1~~~4~~~5 &  ~1~~~6~~~7 &  ~1~~~8~~~9 &  ~1~10~11      &  ~1~12~13      &  ~1~14~15\\
& ~4~10~14        & ~2~~~9~11    & ~2~~~8~10     & ~2~13~15        & ~2~12~14       & ~2~~~4~~~6 & ~2~~~5~~~ 7\\
\mathcal{E}
& ~5 ~~~8 ~13   & ~3~~~8~11    & ~3~~~9~10     & ~3~~~4~~~7  & ~3~~~5~~~6 & ~3~12~15       & ~3~13~14\\
& ~6~~~9~15     & ~6~11~13       & ~4~~~9~13        & ~5~11~14        & ~4~~~8~12    & ~5~~~9~12    & ~4~~~9~13\\
& ~7~11~12        & ~7~10~13       & ~5~10~15        & ~6~10~12        & ~7~~~9~14    & ~7~~~8~15    & ~6~~~8~14\\
  
\end{array}
$
\end{center}
\end{small}
Consider the following permutations on $X$:
\begin{small}
\\$
\begin{array}{l l}
 \pi_1=(5~~13~~3~~14~~4~~12~~2~~15)(6~~10)(7~~11), & \pi'_1=(2~~10)(3~~11)(4~~8~~14)(5~~9~~15)(6~~12)(7~~13), \\
 \pi_2=(4~~10~~14~~5~~11~~15), & \pi'_2=(15~~11~~5~~14~~10~~4),\\
 \pi_3=(2~~10~~15~~13~~9~~7~~3~~11~~14~~12~~8~~6), & \pi'_3=(2~~6)(3~~7)(4~~8)(5~~9)(10~~15~~13~~11~~14~~12),\\
 \pi_4=(2~~4~~6)(3~~5~~7), & \pi'_4=(6~~4~~2)(7~~5~~3),\\
 \pi_5=(2~~12~~4~~14~~6)(3~~13~~5~~15~~7), & \pi'_5=(15~~6~~2)(3~~14~~7)(4~~13)(5~~12)(8~~10)(9~~11), \\ 
\pi_6=(15~~13~~3)(2~~14~~12)(4~~10~~6)(5~~11~~7), & \pi'_6=(15~~3~~13)(2~~12~~14)(4~~6~~10)(7~~11~~5).\\
\end{array}
$
\end{small}

It is checked by computer programming that

$
\begin{array}{l c l}
|\mathcal{B}\cap \pi_1(\mathcal{D})\cap \pi'_1(\mathcal{D})|_F=1, & |\mathcal{B}\cap \pi_2(\mathcal{B})\cap \pi'_2(\mathcal{B})|_F=2, & |\mathcal{B}\cap \pi_3(\mathcal{D})\cap \pi'_3(\mathcal{D})|_F=3, \\
 |\mathcal{E}\cap \pi_4(\mathcal{E})\cap \pi'_4(\mathcal{E})|_F=4, &  |\mathcal{B}\cap \pi_5(\mathcal{C})\cap \pi'_5(\mathcal{C})|_F=6, & |\mathcal{A}\cap \pi_6(\mathcal{A})\cap \pi'_6(\mathcal{A})|_F=7.\\
\end{array}
$

with flower at $1$.
\end{proof}
\begin{theorem}{\label{J^3_F(9)}}
$I^3(9)\setminus \{4,6,7,9,11,34,35\}\subseteq J^3_F(9)$.
\end{theorem}
\begin{proof}
By Theorem \ref{STS(3v+1) Construction} for $r=3$, and Theorems \ref{3way-Latin} and \ref{J^3_F(3)}, $[12,31]\cup \{39,48\}\subseteq J^3_F(9)$. By Theorem \ref{STS(2v+1) Construction} for $r=4$, and Theorem \ref{J^3_F(4)}, $\{5,8,10,33\}\subseteq J^3_F(9)$. By Lemma \ref{0-in-J^3_F}, $0\in J^3_F(9)$  and by Corollary \ref{cor-k>(2r(r-3)/3)}, $\{36,37,38,40,42\}\subseteq J^3_F(9)$. In \cite{3STS}, the $STS(19)$s which are given in Table $3-2$ have $3$-way flower intersection number $32$ with flower at $19$. Let $X=\{1,2,\ldots,19\}$ and $(X,\mathcal{A})$ be the following $STS(19)$.
\begin{center}
\begin{small}
$
\begin{array}{c c c c c c c c c c}
~1~~2~~3 & ~2~~9~14 & ~3~~8~12 & ~4~~8~19 & ~7~~8~10 & ~1~18~19 & ~4~13~16 & ~6~14~19 & ~9 ~10 ~18 & 10~17~19  \\
~1~~4~~5 & ~2~~5~15  & ~3~~9~19 & ~5~~6~10 & ~7~~~9~13 & ~2~11~13 & ~4~15~17 & ~7~15~19 & ~9~11~15 &  11~14~17  \\
~1~~6~~7 & ~2~~6~18  & ~3~~5~17 & ~5~~7~14 & ~1~10~11 & ~2~12~19 & ~5~11~12 & ~7~17~18 & ~~9~12~17 &  11~16~19  \\
~1~~8~~9 & ~2~~7~16  & ~3~~6~15 & ~5~~8~18 & ~1~12~13 & ~3~10~13 & ~5~13~19 & ~8~13~14 & 10~12~15 & 12~14~18  \\
~4~~6~~9 & ~2~~8~17  & ~3~~7~11 & ~5~~9~16 & ~1~14~15 & ~3~16~18 & ~6~12~16 & ~8~15~16 & 10~14~16 & 13~15~18  \\ 
~2~~4~10  & ~3~~4~14 & ~4~~7~12 & ~6~~8~11 & ~1~16~17 & ~4~11~18 & ~6~13~17 &&&\\
\end{array}
$
\end{small}
\end{center}
Consider the following permutations on $X$:
\\
\begin{small}
$\pi_1=(4~~16)(5~~17)(10~~8~~12~~6~~14~~2~~18)(11~~9~~13~~7~~15~~3~~19), \\
\pi'_1=(2~~10)(3~~11)(12~~14)(13~~15)(17~~7~~9~~5~~18~~16~~6~~8~~4~~19), \\
\pi_2=(6~~14)(7~~15)(8~~12~~4~~16~~2~~18)(9~~13~~5~~17~~3~~19), \\ \pi'_2=(14~~6~~4~~2~~10~~16~~8~~19~~15~~7~~5~~3~~11~~17~~9~~18), \\
\pi_3=(2~~8~~10~~18~~4~~12~~6)(3~~9~~11~~19~~5~~13~~7), \\
\pi'_3=(2~~6~~4~~14~~9~~13~~11)(3~~7~~5~~15~~8~~12~~10)(16~~19~~17~~18).\\
$
\end{small}
It is checked by computer programming that $|\mathcal{A}\cap \pi_i(\mathcal{A})\cap \pi'_i(\mathcal{A})|_F=i$ for $i=1,2,3$ with flower at $1$.
\end{proof}
\begin{theorem}{\label{J^3_F(10)}}
$J^3_F(10)=I^3_F(10)$.
\end{theorem}
\begin{proof}
By Corollary \ref{cor-k>(2r(r-3)/3)}, $[47,52]\cup \{54,60\}\subseteq J^3_F(10)$. By Theorem \ref{STS(3v+3) Construction} for $r=3$, and Theorems \ref{3way-Latin} and \ref{J^3_F(4)}, $[0,44]\subseteq J^3_F(10)$ and in Theorem \ref{STS(3v) Construction} for $r=3$, with the aim of Theorems \ref{3STS}, \ref{Latin J''^3(2n+1)} and \ref{J^3_F(3)}, we can assume that $b=33$, $a_1=1$, $a_2=7$ and $a=4$, so $45\in J^3_F(10)$. In Lemma $2.8$ of \cite{3STS}, the three systems with $3$-way intersection number $56$, actually have $3$-way flower intersection number $46$, so $46\in J^3_F(10)$.
\end{proof}

Let us call the following Latin square, $L$, of order $16$ on symbols $0$, $1$, \ldots, $15$, \emph{special} of order $16$, where $A$ and $C$ are Latin squares of order $4$ on symbol set $\{4, 5, 6 ,7\}$ and $B$ is a Latin square of order $8$ on symbol set $\{8, 9, \ldots , 15\}$.
\begin{footnotesize}
\begin{center}$L=$
\begin{tabular}{|c c | c c | c c | c c | c c | c c | c c | c c|}
\hline
\multicolumn{1}{|c}{$0$} & $1$  & $2$ & $3$ & \multicolumn{4}{c|}{\multirow{4}{*}{$A$}}& \multicolumn{8}{c|}{\multirow{8}{*}{$B$}}                                            \\
\multicolumn{1}{|c}{$1$} & $0$  & $3$ & $2$  & \multicolumn{4}{c|}{}                 & \multicolumn{8}{c|}{}                                                              \\ \cline{1-4}
\multicolumn{1}{|c}{$2$} & $3$  & $0$ & $1$ & \multicolumn{4}{c|}{}                 & \multicolumn{8}{c|}{}                                                              \\
\multicolumn{1}{|c}{$3$} & $2$  & $1$ & $0$ & \multicolumn{4}{c|}{}                 & \multicolumn{8}{c|}{}                                                              \\ \cline{1-8}
\multicolumn{4}{|c|}{\multirow{4}{*}{$A^T$}}& $0$     & $1$     & $2$    & $3$      & \multicolumn{8}{c|}{}                                                              \\
\multicolumn{4}{|c|}{}                      & $1$     & $0$     & $3$    & $2$      & \multicolumn{8}{c|}{}                                                              \\ \cline{5-8}
\multicolumn{4}{|c|}{}                      & $2$     & $3$     & $0$    & $1$      & \multicolumn{8}{c|}{}                                                              \\
\multicolumn{4}{|c|}{}                      & $3$     & $2$     & $1$    & $0$      & \multicolumn{8}{c|}{}                                                              \\ \hline
\multicolumn{8}{|c|}{\multirow{8}{*}{$B^T$}}                                        &\multicolumn{1}{c}{$0$} & $1$ & $2$ & $3$ & \multicolumn{4}{c|}{\multirow{4}{*}{$C$}}\\
\multicolumn{8}{|c|}{}                                                              & \multicolumn{1}{c}{$1$} & $0$ & $3$ & $2$  & \multicolumn{4}{c|}{}                   \\ \cline{9-12}
\multicolumn{8}{|c|}{}                                                              & \multicolumn{1}{c}{$2$} & $3$ & $0$ & $1$  & \multicolumn{4}{c|}{}                   \\
\multicolumn{8}{|c|}{}                                                              & \multicolumn{1}{c}{$3$} & $2$ & $1$ & $0$  & \multicolumn{4}{c|}{}                   \\ \cline{9-16} 
\multicolumn{8}{|c|}{}                                                              & \multicolumn{4}{c|}{\multirow{4}{*}{$C^T$}}& $0$    & $1$    & $2$    & $3$       \\
\multicolumn{8}{|c|}{}                                                              & \multicolumn{4}{c|}{}                      & $1$    & $0$    & $3$    & $2$       \\ \cline{13-16}
\multicolumn{8}{|c|}{}                                                              & \multicolumn{4}{c|}{}                      & $2$    & $3$    & $0$    & $1$       \\
\multicolumn{8}{|c|}{}                                                              & \multicolumn{4}{c|}{}                      & $3$    & $2$    & $1$    & $0$       \\ \hline
\end{tabular}
\end{center}
\end{footnotesize}
Denote by $M$ the set of integers $m$ for which there exists a collection of three special Latin squares of order $16$ which pairwise agree in exactly $m$ of the $112$ cells above the $2\times 2$ diagonal blocks.
\begin{lemma}{\label{Latin-16-Intersection}}
$M=[16,97]\cup \{100,103,112\}$.
\end{lemma}
\begin{proof}
Let $m\in M$, write $m=a+b+c+d$ where $a,c\in J'^3(4)$, $b\in J'^3(8)$ and $d=16$ (the number of cells which contain elements $2$ and $3$). By Theorem \ref{3way-Latin} $J'^3(4)=\{0,1,4,16\}$ and $J'^3(8)=[0,49]\cup \{52,55,64\}$, so $m\in [16,97]\cup \{100,103,112\}$. 
\end{proof}
\begin{theorem}{\label{J^3_F(15)}}
$J^3_F(15)=I^3_F(15)$.
\end{theorem}
\begin{proof}
Let $M$ be the set of integers $m$ for which there exists a collection of three special Latin squares of order $16$, $L^{(n)}_{ij}$ for $n=1,2,3$, which pairwise agree in exactly $m$ of the $112$ cells above the $2\times 2$ diagonal blocks. We construct a collection of three $STS(31)$s with $3$-way flower intersection number $h$ for $h\in [16,125]\cup \{128,131,134,140\}$. Write $h=\ell+m$ where $\ell\in J^3_F(7)$ and $m\in M$. Let $X_1=\{\infty_i:~1\leq i\leq 15\}$, $X_2=\{1,2,\ldots,16\}$ and $(X_1,\mathcal{B}_n)$ for $n=1,2,3$ be three $STS(15)$s with $3$-way intersection number $\ell +7$ where seven of these will constitute the flower at $\infty_1$. It is not hard to check that $(X_1\cup X_2,\mathcal{B}_n\cup \mathcal{C}_n)$ for $n=1,2,3$ are three $STS(31)$s, where $\mathcal{C}_n=\{\{\infty_s,i,j\}|~\ell^{(n)}_{ij}=s,~1\leq i < j\leq 16,~1\leq s\leq 15\}$ which have $3$-way intersection number $\ell +m+15$, fifteen of these will constitute the flower at $\infty_1$. So $h=\ell +m\in J^3_F(15)$. By Theorem \ref{J^3_F(7)}, $\ell \in [0,8]\cup [10,13]\cup \{16,22,28\}$ and by Lemma \ref{Latin-16-Intersection}, $m\in [16,97]\cup \{100,103,112\}$. So $[16,125]\cup \{128,131,134,140\}\subseteq J^3_F(15)$. By Theorem \ref{STS(2v+1) Construction} for $r=7$, $[0,15]\subseteq J^3_F(15)$. Existence of the remaining flower intersection numbers is guaranteed by Corollary \ref{cor-k>(2r(r-3)/3)}.
\end{proof}
\begin{theorem}{\label{J^3_F(24)}}
$I^3_F(24)\setminus [1,15]\subseteq J^3_F(24)$.
\end{theorem}
\begin{proof}
By Lemma \ref{0-in-J^3_F}, $0\in J^3_F(24)$ and by Corollary \ref{cor-k>(2r(r-3)/3)}, $[336,360]\cup \{362,368\}\subseteq J^3_F(24)$. There exists a $\{4\}$-GDD of type $3^46^2$ \cite{4-GDD-1997}. All input designs required in Theorem \ref{PBD Construction} to achieve remaining intersection numbers, is guaranteed by Theorems \ref{J^3_F(3)}, \ref{J^3_F(4)} and \ref{J^3_F(6)}.
\end{proof}
\begin{theorem}{\label{J^3_F(60)}}
$J^3_F(60)=I^3_F(60)$.
\end{theorem}
\begin{proof}
By Corollary \ref{cor-k>(2r(r-3)/3)}, $[2280,2352]\cup \{2354,2360\}\subseteq J^3_F(60)$. There exists a $\{4\}$-GDD of type $6^{10}$ \cite{4-GDD-2014}. All input designs required in Theorem \ref{PBD Construction} to achieve remaining intersection numbers, is guaranteed by Theorems \ref{J^3_F(4)} and \ref{J^3_F(6)}.
\end{proof}
\begin{theorem}{\label{J^3_F(132)}}
$J^3_F(132)=I^3_F(132)$.
\end{theorem}
\begin{proof}
By Corollary \ref{cor-k>(2r(r-3)/3)}, $[11352,11520]\cup \{11522,11528\}\subseteq J^3_F(132)$. There exists a $\{4\}$-GDD of type $9^{12}24^{1}$ \cite{4-GDD-2013}. All input designs required in Theorem \ref{PBD Construction} to achieve remaining intersection numbers, is guaranteed by Theorems \ref{J^3_F(4)}, \ref{J^3_F(9)} and \ref{J^3_F(24)}.
\end{proof}
%%%%%%%%%%%%%%%%%%%%%%%%%%%%%%%%%%%%%%%%%%%%%%%%%%%%%%%%%%%%
%%%%%%%%%%%%%%%%%%%%%%%%%%%%%%%%%%%%%%%%%%%%%%%%%%%%%%%%%%%%
%%%%%%%%%%%%%%%%%%%%%%%%%%%%%%%%%%%%%%%%%%%%%%%%%%%%%%%%%%%%
%%%%%%%%%%%%%%%%%%%%%%%%%%%%%%%%%%%%%%%%%%%%%%%%%%%%%%%%%%%%
\section{Main results}
Now, we are in position to present the main theorem.
\begin{theorem}\emph{\textbf{(Main Theorem)}}

Let $S^3[m]$ denote the set of non-negative integers less than or equal to $m$, with the exception of $m-1$, $m-2$, $m-3$, $m-4$, $m-5$ and $m-7$ and let $I^3_F(n)=S^3[\frac{2n(n-1)}{3}]$.
\\For $n\equiv 0,1$ (mod $3$), $n\geq 10$ but $n\neq 24$, $J^3_F(n)=I^3_F(n)$. $J^3_F(3)=\{4\}$, $J^3_F(4)=\{0,8\}$, $[0,5]\cup \{7,20\}\subseteq J^3_F(6)\subseteq [0,7]\cup \{20\}$, $[0,8]\cup [10,13]\cup \{16,22,28\}\subseteq J^3_F(7)\subseteq [0,16]\cup \{22,28\}$, $I^3_F(9)\setminus \{4,6,7,9,11,34,35\}\subseteq J^3_F(9)$ and $I^3_F(24)\setminus [1,15]\subseteq J^3_F(24)$.
\end{theorem}
\begin{proof}
The proof is based on recursive constructions where expressed before. For any admissible $n$, consider the following five cases. It is worth mentioning that in all cases, to construct $STS(2n+1)$, recursive constructions may use small $STS$s of order $2m+1$ where $n$ and $m$ are not congruent modulo $9$.     
\begin{enumerate}
\item $n\equiv 1,4$ (mod $9$)

$J^3_F(1)$, $J^3_F(4)$ and $J^3_F(10)$ have been obtained in Theorems \ref{J^3_F(1)}, \ref{J^3_F(4)} and \ref{J^3_F(10)}, so let $n\geq 13$. Let $n=3r+1$, clearly $r$ is admissible, since $n\equiv 1,4$ (mod $9$). All required objects in Theorem \ref{STS(3v) Construction}, is guaranteed by Theorems \ref{3STS} and \ref{Latin J''^3(2n+1)} and the $3$-way flower intersection numbers of $STS(2r+1)$.
\item $n\equiv 0,3$ (mod $9$)

$J^3_F(3)$ and $J^3_F(9)$ have been obtained in Theorems \ref{J^3_F(3)} and \ref{J^3_F(9)}, so let $n\geq 12$. Let $n=3r$, clearly $r$ is admissible, since $n\equiv 0,3$ (mod $9$). All required objects in Theorem \ref{STS(3v+1) Construction}, is guaranteed by Theorem \ref{3way-Latin} and the $3$-way flower intersection numbers of $STS(2r+1)$.
\item $n\equiv 7$ (mod $9$)

$J^3_F(7)$ has been obtained in Theorem \ref{J^3_F(7)}, so let $n\geq 16$. Let $n=3r+1$, clearly $r\equiv 0,2$ (mod $3$) (it means $r+1$ is admissible), since $n\equiv 7$ (mod $9$). All required objects in Theorem \ref{STS(3v+3) Construction}, is guaranteed by Theorem \ref{3way-Latin} and the $3$-way flower intersection numbers of $STS(2r+3)$.
\item $n\equiv 15$ (mod $18$) (It means $n=9k+6$, where $k$ is odd).

$J^3_F(15)$ has been obtained in Theorem \ref{J^3_F(15)}, so let $n\geq 33$. Let $n=2r+1$, clearly $r$ is admissible, since $n\equiv 15$ (mod $18$). All required objects in Theorem \ref{STS(2v+1) Construction}, is guaranteed by the $3$-way flower intersection numbers of $STS(2r+1)$.
\item $n\equiv 6$ (mod $18$) (It means for $n=9k+6$, where $k$ is even).

$J^3_F(6)$, $J^3_F(24)$, $J^3_F(60)$ and $J^3_F(132)$ have been obtained in Theorems \ref{J^3_F(6)}, \ref{J^3_F(24)}, \ref{J^3_F(60)} and \ref{J^3_F(132)}, so let $n\geq 42$ and $n\notin \{60,132\}$. Let $n=6t$, clearly $t\equiv 1$ (mod $3$), $t\geq 7$ and $t\notin \{10,22\}$. By Theorem \ref{6t-PBD}, there exists a $B(\{6,t\},1,6t)$ and since $t$ and $6$ are admissible, by Theorem \ref{STS-from-PBD}, there exists an $STS(12t+1)$. All required objects in Theorem \ref{PBD Construction}, is guaranteed by Theorem \ref{J^3_F(6)} and the $3$-way flower intersection numbers of $STS(2t+1)$.
\end{enumerate} 
\end{proof}

 %%%%%%%%%%%%%%%%%%%%%%%%%%%%%%%%%%%%%%%%%%%%%%%%%%%%%%%%%%%
%%%%%%%%%%%%%%%%%%%%%%%%%%%%%%%%%%%%%%%%%%%%%%%%%%%%%%%%%%%
%%%%%%%%%%%%%%%%%%%%%%%%%%%%%%%%%%%%%%%%%%%%%%%%%%%%%%%%%%%
%%%%%%%%%%%%%%%%%%%%%%%%%%%%%%%%%%%%%%%%%%%%%%%%%%%%%%%%%%%
%%%%%%%%%%%%%%%%%%%%%%%%%%%%%%%%%%%%%%%%%%%%%%%%%%%%%%%%%%%
%%%%%%%%%%%%%%%%%%%%%%%%%%%%%%%%%%%%%%%%%%%%%%%%%%%%%%%%%%%
\acknowledgements
The authors of this article are thankful to Mr. Bakhshi and Mr. Soltani for helping us to find some small cases with the computer programming. The first author would like to thank the Iran National Science Foundation (INSF) for the financial support of the project.

%%%%%%%%%%%%%%%%%%%%%%%%%%%%%%%%%%%%%%%%%%%%%%%%%%%%%%%%%%%
%%%%%%%%%%%%%%%%%%%%%%%%%%%%%%%%%%%%%%%%%%%%%%%%%%%%%%%%%%%
%%%%%%%%%%%%%%%%%%%%%%%%%%%%%%%%%%%%%%%%%%%%%%%%%%%%%%%%%%%

%%%%%%%%%%%%%%%%%%%%%%%%%%%%%%%%%%%%%%%%%%%%%%%%%%%%%%%%%%%%%%
%%%%%%%%%%%%%%%%%%%%%%%%%%%%%%%%%%%%%%%%%%%%%%%%%%%%%%%%%%%%%%
%%%%%%%%%%%%%%%%%%%%%%%%%%%%%%%%%%%%%%%%%%%%%%%%%%%%%%%%%%%%%%
%%%%%%%%%%%%%%%%%%%%%%%%%%%%%%%%%%%%%%%%%%%%%%%%%%%%%%%%%%%%%%

%\nocite{*}
%\bibliographystyle{abbrvnat}
% use the following instead if you encounter problems 
%\bibliographystyle{alpha}
%\bibliography{flower-intersection-dmtcs}
\label{sec:biblio}

\end{document}